\documentclass[12pt,epsfig,amsfonts]{amsart} 
\usepackage{amsmath,amsthm,amssymb,amscd,epsfig,lmodern}
\usepackage{graphicx}
\usepackage{mathrsfs}
\usepackage[T1]{fontenc}
\usepackage[utf8]{inputenc}
\usepackage{draftwatermark}

\newtheorem*{theorema}{Theorem A}
\newtheorem*{theoremb}{Theorem B}

\newtheorem*{theoremd}{Lemma (the Realization Lemma)}
\newtheorem*{theoremshinoda}{Theorem 1}

\newtheorem*{theoremcont}{Theorem 2}

\newtheorem*{cor3}{Corollary 3}

\newtheorem{lemma}{Lemma}[section]

\begin{document}
\author{Mao Shinoda and Hiroki Takahasi}


\address{Department of Mathematics,
Keio University, Yokohama,
223-8522, JAPAN} 
\email{shinoda-mao@z3.keio.jp}
\email{hiroki@math.keio.ac.jp}
\urladdr{\texttt{http://www.math.keio.ac.jp/~hiroki/}}

\thanks{
M. S. was partially supported by the Grant-in-Aid for JSPS Research Fellows.
H. T. was partially supported by 
the Grant-in-Aid for Young Scientists (A) of the JSPS 15H05435 and
the Grant-in-Aid for Scientific Research (B) of the JSPS 16KT0021.
Both authors were partially supported by the JSPS Core-to-Core Program ``Foundation
of a Global Research Cooperative Center in Mathematics focused on Number Theory and Geometry''.}

\title[ Lyapunov optimization]  
{Lyapunov optimization for non-generic\\ one-dimensional expanding Markov  maps}

\date{\today}

\begin{abstract}
For a non-generic, yet dense subset of $C^1$ expanding Markov maps of the interval 
we prove the existence 
of uncountably many Lyapunov optimizing measures which are ergodic, fully supported and have positive entropy.
These measures are equilibrium states for some H\"older continuous potentials.
We also prove the existence of another non-generic dense subset for which  
the optimizing measure is unique and supported on a periodic orbit.
A key ingredient is a new $C^1$ perturbation lemma which allows us to interpolate 
between expanding Markov maps and the shift map on a finite number of symbols.
\end{abstract}
\maketitle

\section{Introduction}
Ergodic optimization aims to describe
invariant probability measures of a dynamical system which optimize the integral of
a given ``performance'' function. 
In its basic form, given a continuous self-map $T$ of a compact metric space $M$
and a real-valued continuous function $\phi$ on $M$
one seeks {\it $\phi$-minimizing measures}, i.e., measures in $\mathcal M(T)$ which attain the infimum
$$\inf\left\{\int \phi d\mu\colon\mu\in\mathcal M(T)\right\},$$
where $\mathcal M(T)$ denotes the set of $T$-invariant Borel probability measures.
A root of this theory is found in the works of Mather \cite{Mat91} and Ma\~n\'e \cite{Man96,Man97}
on the dynamics of the Euler-Lagrange flow:
orbits with prescribed properties can be obtained by considering ergodic invariant probability measures
which minimize the integral of the Lagrangian, 
and the orbits are obtained as typical points for these measures
(called {\it action minimizing measures}, see Sorrentino \cite{Sor15}).

A general belief is that the minimizing measure is unique and supported on a periodic orbit,
for ``most'' performance functions.
The meaning of ``most'' is the genericity in the sense of Baire's Category Theorem, 
  the typicality in the sense of the Lebesgue measure in parameter space, and so on.
For instance,  see Ma\~n\'e's conjecture \cite{Man97}, and that of Yuan and Hunt \cite{YuaHun99}
after numerical observations by Hunt and Ott  \cite{HunOtt96a,HunOtt96b}.
For the doubling map $x\mapsto 2x$ (mod $1$) and a parametrized performance function
$\cos(2\pi(x-\theta))$ $(\theta\in\mathbb R)$, Bousch  \cite{Bou00} proved
that for all $\theta$ the minimizing measure is unique, and is supported on a subset of a semi-circle.
Further, he proved that for Lebesgue almost every $\theta$ the corresponding minimizing measure is supported on a periodic orbit.
Contreras {\it et al} \cite{ConLopThi01} studied $C^1$ expanding maps of the circle and performance functions in the Banach space $C^{0,\alpha}$
 of $\alpha$-H\"older continuous functions. They considered its subspace 
 consisting of those $\phi\in C^{0,\alpha}$ for which $|\phi(x)-\phi(y)|=o(|x-y|^\alpha)$, and showed that 
 for an open dense subset of this subspace the minimizing measure 
 is unique and supported on a periodic orbit.
 For transitive and ``expanding'' dynamical systems,
 the uniqueness of minimizing measure was established
for open dense subsets of suitable separable Banach spaces,
 see Bousch  \cite{Bou01} and Contreras \cite{Con16}.
 Quas and Siefken \cite{QuaSie12} proved the uniqueness of minimizing measure
for the one-sided full shift and an open dense subset of a certain non-separable Banach space.
In all these three results, the unique minimizing measure is supported on a periodic orbit.
For generic functions in the space $C^0$ of continuous functions 
the uniqueness of minimizing measure still holds, but the measure is fully supported,
see Bousch and Jenkinson \cite{BouJen02}, Jenkinson \cite{Jen06}.
Br\'emont \cite{Bre08} proved that for generic $C^0$ functions, any minimizing measure has zero entropy.

In these developments, the non-uniqueness of minimizing measure was considered somewhat irrelevant,
as it is considered to be a non-generic property. However, the non-uniqueness of action-minimizing measure is a universal phenomenon
(the action minimizing measure is, in a sense, a generalization of the KAM tori \cite{Sor15}).
In the context of the thermodynamic formalism, the study of the non-uniqueness of minimizing measure has a connection 
with the asymptotic behavior of Gibbs measures as the temperature drops to zero, see Baraviera, Leplaideur and Lopes
\cite{BarLepLop13}.
Despite its importance,
the non-uniqueness of minimizing measure has not yet received an adequate deal of attention.

In this paper we treat expanding Markov maps of the interval, and show that
the uniqueness of Lyapunov optimizing measure fails in a severe way.
Let $p\geq2$ be an integer and $\{\alpha_i\}_{i=0}^{p-1}$,
 $\{\beta_i\}_{i=0}^{p-1}$ sequences in $[0,1]$ such that
\begin{equation*}
0=\alpha_0<\beta_0<\alpha_1<\beta_1<\alpha_2<\cdots<\beta_{p-1}=1.\end{equation*}
Put $X=\bigcup_{i=0}^{p-1}[\alpha_i,\beta_i]$.
Denote by $\mathscr{E}$
the set of $C^1$ maps on $X$ such that
the following holds for every $f\in \mathscr{E}$:

\begin{itemize}
\item[(E1)] $f$ maps each interval $[\alpha_i,\beta_i]$, $i\in\{0,\ldots,p-1\}$
diffeomorphically onto $[0,1]$;

\item[(E2)]
there exist constants $c>0$, $\lambda>1$ such that for every $n\geq1$ and
every $x\in\bigcap_{k=0}^{n-1} f^{-k}\left(X\right)$,
 $|Df^n(x)|\geq c\lambda^n$.  \end{itemize}
We endow $\mathscr{E}$ with the $C^1$ topology given by the norm
$$\|f\|_{C^1}=\sup_{x\in X}|f(x)|+\sup_{x\in X}|Df(x)|.$$
The space $\mathscr{E}$ is an open subset of $C^1$ functions on $X$ (See Lemma \ref{standard} to check the openness of (E2)),
and hence becomes a (non-complete) Baire space.

Define
$$\Lambda(f)=\bigcap_{n\geq0} f^{-n}\left(X\right).$$
Restricting $f$ to $\Lambda(f)$
we obtain a dynamical system which is
also denoted by $f$ with a slight abuse of notation.
Then $\Lambda(f)$ is a Cantor set with a Markov partition
 given by the collection $\{[\alpha_i,\beta_i]\}_{i=0}^{p-1}$ of intervals which topologically conjugates $f$ to
the left shift on $p$ symbols.


Put $$\chi_{\rm inf}(f)=\inf\left\{\int\log|Df|d\mu\colon\mu\in\mathcal M(f)\right\}$$
and
$$\chi_{\rm sup}(f)=\sup\left\{\int\log|Df|d\mu\colon\mu\in\mathcal M(f)\right\}.$$
Since $\mathcal M(f)$ is compact and $\nu\in\mathcal M(f)\mapsto\int\log|Df|d\nu$ is continuous,
the infimum and supremum are attained. 
A measure $\mu\in\mathcal M(f)$ is {\it Lyapunov minimizing} if
$$\int \log|Df|d\mu=\chi_{\inf}.$$
 {\it Lyapunov maximizing measures} is defined similarly, with $\chi_{\rm inf}$ replaced by $\chi_{\rm sup}$.


The notion of Lyapunov optimizing measures was introduced by Contreras {\it et al} \cite{ConLopThi01}. 
They showed that for an open dense subset of  the space $\bigcup_{\beta>\alpha} C^{1+\beta}$ of expanding maps
of the circle
in the $C^{1+\alpha}$ topology,
the Lyapunov minimizing measure is unique and supported on a periodic orbit.
For a generic $C^1$ expanding map of the circle, Jenkinson and Morris \cite{JenMor08} proved that
the Lyapunov minimizing measure is unique and has zero entropy.
See Morita and Tokunaga \cite{MorTok13}, Tokunaga \cite{Tok15} for extensions of the results
of \cite{JenMor08}  to higher dimension.
With the method of 
\cite{JenMor08} one can
show that for generic maps in $\mathscr{E}$ the Lyapunov minimizing measure is unique,
and it is fully supported, has zero entropy. 
In the realm of non-genericity the structure of Lyapunov minimizing measures is in contrast.

\begin{theorema}
There exists a dense subset $\mathscr{A}$ of $\mathscr{E}$ such that the following holds
for every $f\in\mathscr{A}$:

\begin{itemize}
\item[-] $\chi_{\rm inf}(f)\neq  \chi_{\rm sup}(f)$;
\item[-] there exist uncountably many Lyapunov minimizing measures of $f$
which are ergodic, fully supported and have positive entropy;
\item[-]  $\log|Df|$ is not H\"older continuous.
\end{itemize}\end{theorema}

Theorem A has been inspired by the following result of the first-named author \cite{Shi17}
on the ergodic optimization for the subshift of finite type.
For an integer $p\geq2$ let $\Sigma_p=\{0,1,\ldots,p-1\}^{\mathbb Z_{\geq0}}$ denote the one-sided shift space
on $p$ symbols, endowed with the product topology of the discrete topology on $\{0,1,\ldots,p-1\}$.
Let $\sigma\colon\Sigma_p\circlearrowleft$ denote the left shift:
if $\underline{a}$, $\underline{b}\in\Sigma_p$, $\underline{a}=\{a_n\}_{n\geq0}$,
$\underline{b}=\{b_n\}_{n\geq0}$
and $\sigma(  \underline{a})=
\underline{b}$ then $b_n=a_{n+1}$ holds for every $n\geq0$.
For a $p\times p$ matrix $A=(A_{ij})$ whose each entry is $0$ or $1$, define
$\Sigma_A=\{\{a_n\}_{n\geq0}\in\Sigma_p\colon A_{a_na_{n+1}}=1\ \forall n\geq0\}$.
{\it The subshift of finite type} is the dynamical system
$\sigma_A\colon \Sigma_A\circlearrowleft$ 
given by $\sigma_A=\sigma|_{\Sigma_A}$.
We say $\sigma_A$ is {\it topologically mixing} if there exists an integer $M>0$
such that all the entries of $A^M$ are positive.
Let $C(\Sigma_A)$ denote the space of real-valued continuous functions on $\Sigma_A$ endowed with the supremum norm.
\begin{theoremshinoda}{\rm (\cite[Theorem A]{Shi17})}\label{shinoda}
Let $\sigma_A\colon \Sigma_A\circlearrowleft$ be a topologically mixing subshift of finite type.
There exists a dense subset $\mathscr{C}_{\sigma_A}$ of $C(\Sigma_A)$ such that for every $\phi\in\mathscr{C}_{\sigma_A}$
there exist uncountably many $\phi$-minimizing measures which are ergodic, fully supported and have positive entropy. 
\end{theoremshinoda}

The argument in \cite{Shi17} is to view elements of $\mathcal M(\sigma_A)$ as duals of $C(\Sigma_A)$,
and use the Bishop-Phelps Theorem \cite[Theorem V.1.1]{Isr79} for a continuous path of ergodic Markovian measures. 
Our strategy for proving Theorem A is to use the Markov partition formed by the intervals 
$\{[\alpha_i,\beta_i]\}_{i=0}^{p-1}$ to translate the Lyapunov optimization 
to the ergodic optimization for $\sigma$,
and then appeal to Theorem \ref{shinoda}.
As a result, for every $f\in\mathscr{A}$ the uncountably many minimizing measures in the statement lie on a continuous path 
of ergodic measures which are equilibrium states of $f$ for some H\"older continuous potentials.

We have suppressed small generalizations of Theorem A for the brevity of presentation.
An extension is possible to the case where $f$ is topologically conjugate
the topologically mixing subshift of finite type. Moreover, we can drop the 
conditions $0=\alpha_0$ and $\beta_{p-1}=1$.
Statements analogous to Theorem A hold for Lyapunov maximizing measures.
Since both proofs are identical, we restrict ourselves to Lyapunov minimizing ones.

For $f\in\mathscr{E}$ define a coding map $\pi_f\colon\Sigma_p\to \Lambda(f)$ by
$\pi_f(\underline{a})\in\bigcap_{n\geq0}f^{-n}([\alpha_{a_n},\beta_{a_n}])$
where $\underline{a}=\{a_n\}_{n\geq0}$.
Then
  $\pi_f$ is a homeomorphism satisfying
$\pi_f\circ\sigma= f\circ\pi_f$, and 
for each fixed $\phi\in C(\Sigma_p)$ there is a one-to-one correspondence
between $\phi$-minimizing measures in $\mathcal M(\sigma)$ and $\phi\circ\pi_f^{-1}$-minimizing ones 
in $\mathcal M(f)$. 
Theorem A is proved by combining Theorem 1 and the next lemma
which allows one to realize a perturbation in $C(\Sigma_p)$ as a perturbation
in $\mathscr{E}$.
\begin{theoremd}
Let $f_0\in \mathscr{E}$ be of class $C^2$.
For every $\varepsilon>0$ there exists a neighborhood $\mathscr U$ of
 $\log|Df_0|\circ\pi_{f_0}$ in $C(\Sigma_p)$ 
such that for every $\phi\in \mathscr{U}$ 
there exists $f_\infty\in\mathscr{E}$ such that
$$\|f_0-f_\infty\|_{C^1}\leq\varepsilon\ \ \text{and}\ \ \log|Df_\infty|\circ\pi_{f_\infty}=\phi.$$
 \end{theoremd}

We finish the proof of Theorem A assuming the Realization Lemma.
\begin{proof}[Proof of Theorem A]
For the shift map $\sigma_p\colon\Sigma_p\circlearrowleft$ consider the dense subset 
$\mathscr{C}_{\sigma_p}$ of $C(\Sigma_p)$ in Theorem 1.
Since $C^2$ maps are dense in
$\mathscr{E}$, 
the Realization Lemma 
implies that the set $\widehat{\mathscr{A}}=\{f\in\mathscr{E}\colon \log|Df|\circ\pi_f\in\mathscr{C}_{\sigma_p}\}$
is dense in $\mathscr{E}$.
From Theorem 1, if $f\in\widehat{\mathscr{A}}$ then there exist uncountable many Lyapunov minimizing measures which are ergodic, fully supported and have positive entropy.

Let $\widehat{\mathscr{E}}$ denote the set of elements of $\mathscr{E}$ for which there exist two periodic measures with different 
Lyapunov exponents. Clearly, if $f\in\widehat{\mathscr{E}}$ then  $\chi_{\rm inf}(f)\neq\chi_{\rm sup}(f)$.
Set $\mathscr{A}=\widehat{\mathscr{A}}\cap\widehat{\mathscr{E}}$.
The set $\mathscr{A}$ satisfies the desired properties.
Indeed, since $\widehat{\mathscr{E}}$ is an open dense subset of $\mathscr{E}$
and $\widehat{\mathscr{A}}$ is a dense subset of $\mathscr{E}$,
$\mathscr{A}$ is a dense subset of $\mathscr{E}$.
Let $f\in\mathscr{A}$ and suppose $\log|Df|$ is H\"older continuous.
From the so-called {\it Ma\~n\'e-Conze-Guivarc'h lemma}
(See \cite{Mor09} and the references therein) there exists a continuous function $u$ on $\Lambda(f)$ such that
 $\log|Df|=u-u\circ f+\chi_{\rm inf}$.
Hence $\chi_{\rm inf}(f)=\chi_{\rm sup}(f)$ holds and
 $f\notin\widehat{\mathscr{E}}$, a contradiction.
\end{proof}

For a proof of the Realization Lemma,
we construct a Cauchy sequence $\{ f_n\}_{n\geq 0}$ in 
 $\mathscr{E}$ and obtain $f_\infty$ as a limit of this sequence. 
This construction has two main steps which are carried out in Sect.3.
 First, we construct by induction a sequence of 
 continuous piecewise linear maps  $h_n\colon X\to[0,1]$
 $(n=0,1,\ldots)$.
 Then we perturb each $h_n$ and obtain
the desired sequence $\{f_n\}_{n\geq0}$.
Although the constructions of $\{h_n\}_{n\geq 0}$ and $\{f_n\}_{n\geq 0}$ are rather intuitive,
the difficulty is to ensure that the induction does not halt on the way.
We do this by showing that the sizes of gaps at each step of induction have a definite proportion
(See Lemma \ref{d-ini}).
This is the reason for the assumption of $C^2$ smoothness in the Realization Lemma.
The realization lemma is clearly false for expanding circle maps:
there is no space to absorb differences which stem from the perturbation.
Indeed, our proof exploits the total disconnectedness of the Cantor set.

The Realization Lemma 
 implies that any property of minimizing measures
which holds for a dense subset of functions in $C(\Sigma_p)$ transmits to a dense subset of $\mathscr{E}$.
By the result of Br\'emont  \cite[Proposition 2.1]{Bre08}, for a dense subset of $C(\Sigma_p)$
  the minimizing measure is unique, and it is supported on a periodic orbit.
It follows that for a dense subset of $\mathscr{E}$ the Lyapunov minimizing measure
is unique, and it is supported on a periodic orbit.
Below we give a stronger statement which 
in particular indicates that a version of Theorem A does not hold in the $C^{1+{\rm Lip}}$ topology.

Let $\mathscr{E}_{\rm Lip}$ denote the space of maps in $\mathscr{E}$ with Lipschitz continuous derivative
endowed with the topology given by the norm 
$$\|f\|_{C^{1+{\rm Lip}}}=\|f\|_{C^1}+\sup_{\stackrel{x,y\in X}{x\neq y}}\frac{|Df(x)-Df(y)|}{|x-y|}.$$

 \begin{theoremb}
 There exists an open subset $\mathscr{O}$ of $\mathscr{E}_{\rm Lip}$
 such that for every $f\in\mathscr{O}$ there exists a unique Lyapunov minimizing measure, and 
it is supported on a periodic orbit. In addition, $\mathscr{O}$ 
is a dense subset of $C(\Sigma_p)$. 
\end{theoremb}


The first statement of Theorem B is
a consequence of  the result of Contreras \cite{Con16}.
A proof of the last statement of Theorem B is briefly outlined as follows.
The total disconnectedness of the phase space implies that maps with locally constant derivative
are dense in $\mathscr{E}$ (See Lemma \ref{lc}). If $Df$ is a locally constant function on $\Lambda(f)$,
then $\log|Df|\circ\pi_f$ becomes
Lipschitz continuous with respect to the standard distance on $\Sigma_p$ of {\it any scale}.
By the Realization Lemma and the result of Contreras \cite{Con16}, 
$f$ is approximated by another for which the Lyapunov minimizing measure is unique and 
supported on a periodic orbit. Choosing a distance of sufficiently small scale
relative to the expansion rate of $f$ in (E2) we obtain the Lipschitz continuity of $Df$ (See Sect.\ref{LIP}).



Theorem B has one important consequence on {\it the zero-temperature limit} in the thermodynamic formalism
(See e.g. \cite{BarLepLop13} and the references therein).
For $f\in\mathscr{E}$
define a geometric pressure function $t\in\mathbb R\mapsto \mathscr{P}(t)$ by
$$\mathscr{P}(t)=\sup\left\{h_\mu(f)-t\int \log|Df| d\mu\colon\mu\in\mathcal M(f)\right\},$$
where $h_\mu(f)$ denotes the Kolmogorov-Sina{\u\i} entropy of $(f,\mu)$.
 {\it An equilibrium state for the potential $-t\log|Df|$} is a measure in $\mathcal M(f)$ which attains this supremum.
If $\log|Df|$ is H\"older continuous, then $t\in\mathbb R\mapsto \mathscr P(t)$ is real-analytic, and for every $t\in\mathbb R$ there exists a unique
equilibrium state  for the potential $-t\log|Df|$ \cite{Bow75,Rue04}, which we denote by $\mu_t$.
Lyapunov minimizing measures are obtained by freezing the system: 
any accumulation point of $\{\mu_t\}_{t\in\mathbb R}$ as $t\to+\infty$ is a Lyapunov minimizing measure.
By {\it the zero-temperature limit} we mean the weak* limit of $\{\mu_t\}_{t\in\mathbb R}$
as $t\to+\infty$.
The uniqueness of the Lyapunov minimizing measure implies the existence of the zero-temperature limit.
\begin{cor3}
 For every $f\in\mathscr{O}$ the zero-temperature limit $\displaystyle{\lim_{t\to+\infty}\mu_t}$
exists, and is supported on a periodic orbit.
\end{cor3}
If $\log|Df|$ is a locally constant function on $\Lambda(f)$, then the zero-temperature limit exists, see Br\'emont \cite{Bre03} and 
Leplaideur \cite{Lep05}.
Such maps are dense in $\mathscr{E}$ (See Lemma \ref{lc}).
The non-existence of zero-temperature limit was treated by Chazottes and Hochman
\cite{ChaHoc10}, Coronel and Rivera-Letelier \cite{CorRiv15} in the context of the subshift of finite type.

The rest of this paper consists of three sections.
Sect.2 and Sect.3 are entirely dedicated to a proof of the Realization Lemma.
In Sect.4 we prove Theorem B.

\section{Preliminaries}
In this section we develop fundamental estimates needed for the proofs of the main results.

\subsection{Control of variations}
For an integer $n\geq1$,
by {\it a word of length $n$} we mean an $n$-string of integers in $\{0,1,\ldots,p-1\}$.
For each integer $n\geq1$ let $W_n$ denote the set of words of length $n$. 
For $a_0\cdots a_{n-1}\in W_n$ define
$$[a_0\cdots a_{n-1}]=\{\{b_k\}_{k\geq0}\in\Sigma_p\colon a_k=b_k\text{ for every $k\in\{0,1,\ldots,n-1\}$}\}.$$
For $\phi\in C(\Sigma_p)$ and $a_0\cdots a_{n-1}\in W_n$ define
$$E_{a_0\cdots a_{n-1}}(\phi)=\exp\left(\min\{\phi(\underline{a})\colon\underline{a}\in[a_0\cdots a_{n-1}]\}\right),$$
$$F_{a_0\cdots a_{n-1}}(\phi)=\exp\left(\max\{\phi(\underline{a})\colon\underline{a}\in[a_0\cdots a_{n-1}]\}\right)$$
and 
$$V_{a_0\cdots a_{n-1}}(\phi)=F_{a_0\cdots a_{n-1}}(\phi)-E_{a_0\cdots a_{n-1}}(\phi).$$
Set $$M(\phi)=\sup_{\underline{a}\in \Sigma_p}|\phi(\underline{a})|.$$

\begin{lemma}\label{variation}
For every $\phi,\phi'\in C(\Sigma_p)$, every $n\geq1$ and every $a_0\cdots a_{n-1}\in W_{n}$ the following holds:
$$|E_{a_0\cdots a_{n-1}}(\phi)-E_{a_0\cdots a_{n-1}}(\phi')|\leq V_{a_0\cdots a_{n-1}}(\phi)+M(e^{\phi}-e^{\phi'});$$
$$|F_{a_0\cdots a_{n-1}}(\phi)-F_{a_0\cdots a_{n-1}}(\phi')|\leq V_{a_0\cdots a_{n-1}}(\phi)+M(e^{\phi}-e^{\phi'}).$$
\end{lemma}
\begin{proof}
Let $\underline{a},\underline{a'}\in [a_0\cdots a_{n-1}]$ be such that $E_{a_0\cdots a_{n-1}}(\phi)=e^{\phi(\underline{a})}$
and $E_{a_0\cdots a_{n-1}}(\phi')=e^{\phi'(\underline{a'})}$.
Then
\begin{align*}|E_{a_0\cdots a_{n-1}}(\phi)-E_{a_0\cdots a_{n-1}}(\phi')|&\leq |e^{\phi(\underline{a})}-e^{\phi(\underline{a'})}|+|e^{\phi(\underline{a}')}-e^{\phi'(\underline{a}')}|\\
&\leq  V_{a_0\cdots a_{n-1}}(\phi)+M(e^{\phi}-e^{\phi'}).\end{align*}
A proof of the second inequality is analogous.
\end{proof}

\subsection{Uniform expansion for nearby maps}\label{control}
The next lemma ensures that $\mathscr{E}$ is an open subset of $C^1$ functions on $X$.

\begin{lemma}\label{standard}
Let  $f_0\in \mathscr{E}$.
There exist constants
$\varepsilon_0>0$, $c_0>0$, $\lambda_0>1$ such that the following holds:
 if $f\in\mathscr{E}$
and $\|f_0-f\|_{C^1}\leq\varepsilon_0$, then 
for every $x\in X$ and every $n\geq1$ such that $x,f(x),\ldots,f^{n-1}(x)\in X$,
 $|Df^n(x)|\geq c_0\lambda_0^n$.
\end{lemma}

\begin{proof}
By (E2)
there exist constants  $c>0$, $\lambda>1$ such that
if $x\in X$ and $n\geq1$ are such that $x,f_0(x),\ldots,f_0^{n-1}(x)\in X$
then $|Df_0^n(x)|\geq c\lambda^n$.
Choose an integer $N_0\geq1$ such that $c\lambda^{N_0}\geq3$.
Choose $\varepsilon_0>0$
such that
$|Df^{N_0}(x)|\geq2$ holds for every 
  $f\in\mathscr E$ satisfying  $\|f_0-f\|_{C^1}\leq\varepsilon_0$ and
  $x\in X$ such that $x,f(x),\ldots f^{N_0-1}(x)\in X$. 
For such an $f$, let $x\in X$ and $n\geq1$ be such that
$x,f(x),\ldots,f^{n-1}(x)\in X$.
Write $n=N_0k+l$ where $k,l$ are nonnegative integers with $0\leq l\leq N_0-1$.
Put $$c_0=\frac{1}{2}\min_{\stackrel{y\in \bigcap_{k=0}^{l-1}f_0^{-k}(X)}{1\leq l\leq N_0-1}}|Df_0^l(y)|
\ \text{ and }\ \lambda_0=2^{\frac{1}{N_0}}.$$
Splitting the orbit of $x$ into a concatenation of segments of length $N_0$ and then using the Chain Rule gives
$|Df^n(x)|\geq c_0\lambda_0^n.$
\end{proof}

 
 \subsection{Consequence of bounded distortion}
  Let $a_0\cdots a_{n-1}\in W_n$ and $f\colon X\to[0,1]$ a continuous map. Define 
$$X_{a_0\cdots a_{n-1}}(f)=\{x\in X\colon f^k(x)\in [\alpha_{a_k},\beta_{a_k}]\text{ for every }k\in\{0,1,\ldots,n-1\}\}.$$
For each $i\in\{1,2,\ldots, p-1\}$ define
$$X_{a_0\cdots a_{n-1}\Box_i}(f)=\{x\in  X_{a_0\cdots a_{n-1}}(f)\colon f^n(x)\in(\beta_{i-1},\alpha_{i})   \}$$
and
$$X_{a_0\cdots a_{n-1}\widehat i}(f)=\{x\in  X_{a_0\cdots a_{n-1}}(f)\colon f^n(x)\in[\alpha_{i-1},\beta_{i}]   \}.$$
Denote by $|I|$ the length of a bounded interval $I$. Note that 
$$\left|X_{a_0\cdots a_{n-1}\widehat i}(f)\right|=\left|X_{a_0\cdots a_{n-1}(i-1)}(f)\right|+
\left|X_{a_0\cdots a_{n-1}\Box_i}(f)\right|+\left|X_{a_0\cdots a_{n-1}i}(f)\right|.$$
 Put
$$\Delta^{(i)}_{a_0\cdots a_{n-1}}(f)=\frac{\left|X_{a_0\cdots a_{n-1}\widehat i}(f)\right|}
{        \left|X_{a_0\cdots a_{n-1}\Box_i}(f)\right|           }-1.$$

\begin{lemma}\label{distortion}
Let  $f_0\in \mathscr{E}$ be of class $C^2$.
There exists a constant $K(f_0)>0$
such that for every $i\in\{1,2,\ldots,p-1\}$, every $n\geq1$ and every $a_0\cdots a_{n-1}\in W_n$, 
$$\Delta^{(i)}_{a_0\cdots a_{n-1}}(f_0)\leq K(f_0).$$
\end{lemma}
\begin{proof}
Since $f_0\in \mathscr{E}$ is of class $C^2$
the bounded distortion holds:
there exists $M_0>1$ 
such that
for every $n\geq1$, every $a_0\cdots a_{n-1}\in W_n$ 
and every $x,y\in X_{a_0\cdots a_{n-1}}(f_0)$ we have $$\frac{|Df_0^n(x)|}{|Df_0^n(y)|}\leq M_0.$$ 
Let $i\in\{1,2,\ldots,p-1\}$. For every $a_0\cdots a_{n-1}\in W_n$ 
we have
$$\frac{\left|X_{a_0\cdots a_{n-1}\Box_i}(f_0)\right|}{|X_{a_0\cdots a_{n-1}}(f_0)|}\geq M_0^{-1}
\frac{\left|f_0^n(X_{a_0\cdots a_{n-1}\Box_i}(f_0))\right|}
{\left|f_0^n(X_{a_0\cdots a_{n-1}}(f_0))\right|}=M_0^{-1}(\alpha_i-\beta_{i-1}),$$
and therefore 
\begin{align*}
\frac{|X_{a_0\cdots a_{n-1}(i-1)}(f_0)|+
|X_{a_0\cdots a_{n-1}i}(f_0)|}{|X_{a_0\cdots a_{n-1}}(f_0)|}&\leq\frac{|X_{a_0\cdots a_{n-1}}(f_0)|-
|X_{a_0\cdots a_{n-1}\Box_i}(f_0)|}{|X_{a_0\cdots a_{n-1}}(f_0)|}\\
&\leq 1-M_0^{-1}(\alpha_i-\beta_{i-1}).\end{align*}
Put $K(f_0)= \max_{i\in\{1,2,\ldots,p-1\}}\frac{1-M_0^{-1}(\alpha_i-\beta_{i-1})}{M_0^{-1}(\alpha_i-\beta_{i-1})}.$
These two inequalities yield the desired one.
\end{proof}

\section{On the proof of the Realization Lemma}
In this section we complete the proof of the Realization Lemma.
Throughout this section,
let
 $f_0\in \mathscr{E}$ be of class $C^2$ and
 put $\phi_0=\log|Df_0|\circ\pi_{f_0}$.

\subsection{Construction of a sequence of continuous piecewise linear maps}\label{PL}
Let $\phi\in C(\Sigma_p)$ be sufficiently close to $\phi_0$ and $N\geq2$ an integer.
We construct by induction a sequence $\{h_n(\phi,N)\}_{n\geq0}$ of continuous piecewise linear maps on $X$
which maps each connected component of $X$ bijectively onto $[0,1]$.
In what follows we will write $h_n$ for $h_n(\phi,N)$.

\begin{figure}
\begin{center}
\includegraphics[height=7cm,width=8.3cm]{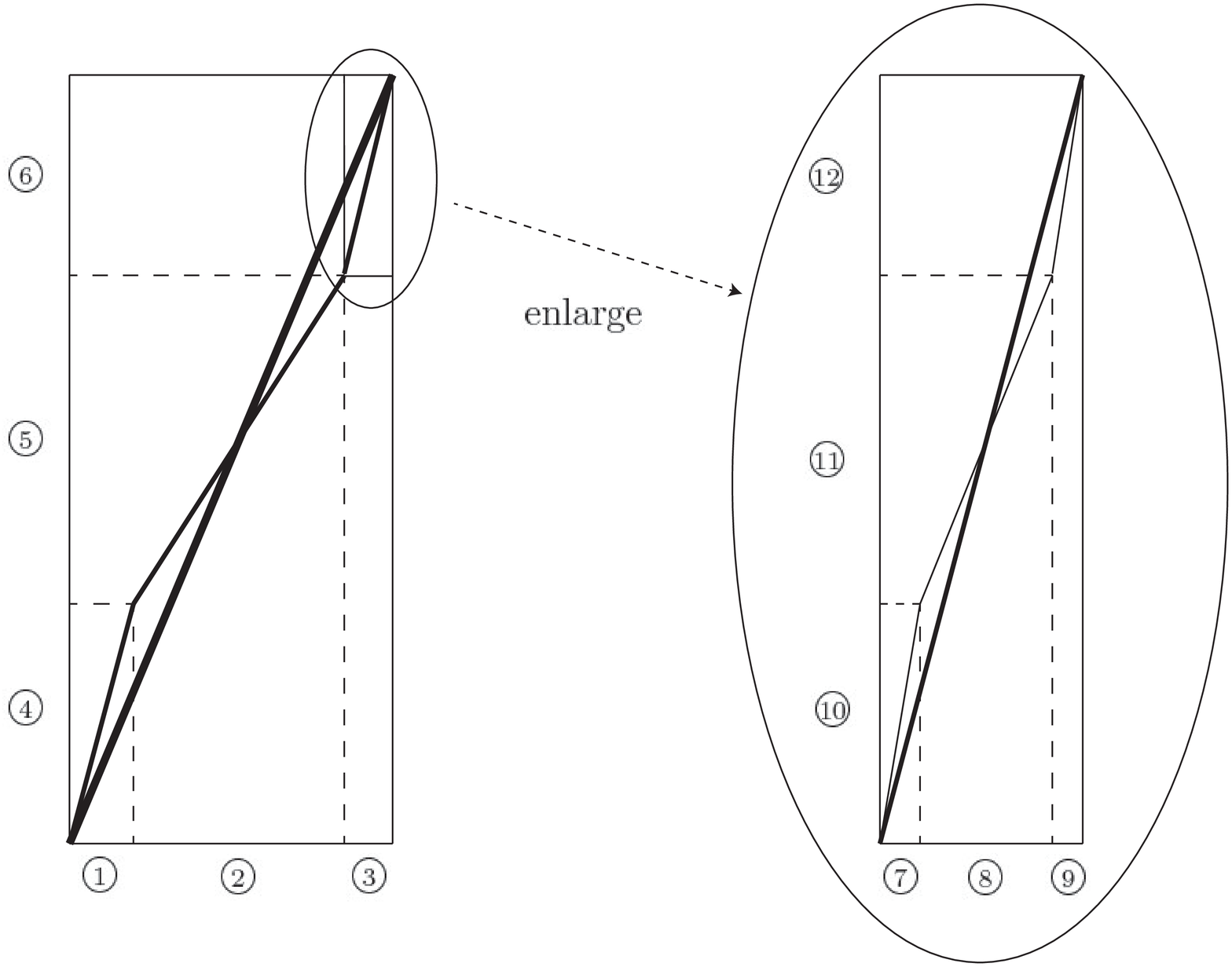}
\caption{The case $p=2$ and $n\geq N+1$ with $a_0\cdots a_{n-1}\in W_{n}$~: 
the construction of $h_{n}|_{X_{a_0\cdots a_{n-1}}(h_{n-1})}$ from $h_{n-1}|_{X_{a_0\cdots a_{n-1}}(h_{n-1})}$, and 
that of $h_{n+1}|_{X_{a_0\cdots a_{n-1}1}(h_{n})}$ from $h_{n}|_{X_{a_0\cdots a_{n-1}1}(h_{n})}$ (in the ellipses).
The left rectangle is $X_{a_0\cdots a_{n-1}}(h_{n-1})\times X_{a_1\cdots a_{n-1}}(h_{n-1})\subset
X\times[0,1]$.
Its diagonal is the graph of $h_{n-1}|_{X_{a_0\cdots a_{n-1}}(h_{n-1})}$, and
the polygonal segment is the graph of $h_{n}|_{X_{a_0\cdots a_{n-1}}(h_{n-1})}$.
The rectangle in the enlarged ellipse is $X_{a_0\cdots a_{n-1}1}(h_{n})\times X_{a_1\cdots a_{n-1}1}(h_{n})$.
Its diagonal is the graph of $h_{n}|_{X_{a_0\cdots a_{n-1}1}(h_{n})}$, and
the polygonal segment is the graph of $h_{n+1}|_{X_{a_0\cdots a_{n-1}1}(h_{n})}$.
The labelled segments are as follows: 
$$\textcircled{\scriptsize1}=X_{a_0\cdots a_{n-1}0}(h_{n});\textcircled{\scriptsize2}=
X_{a_0\cdots a_{n-1}\Box_1}(h_{n}); \textcircled{\scriptsize3}=X_{a_0\cdots a_{n-1}1}(h_{n});$$
$$\textcircled{\scriptsize4}=X_{a_1\cdots a_{n-1}0}(h_{n-1}); 
\textcircled{\scriptsize5}=X_{a_1\cdots a_{n-1}\Box_1}(h_{n-1}); \textcircled{\scriptsize6}=
X_{a_1\cdots a_{n-1}1}(h_{n-1});$$
$$\textcircled{\scriptsize7}=X_{a_0\cdots a_{n-1}10}(h_{n+1}); \textcircled{\scriptsize8}=
X_{a_0\cdots a_{n-1}1\Box_1}(h_{n+1}); 
\textcircled{\scriptsize9}=X_{a_0\cdots a_{n-1}11}(h_{n+1});$$
$$\textcircled{\scriptsize10}=X_{a_1\cdots a_{n-1}10}(h_{n}); 
\textcircled{\scriptsize11}=X_{a_1\cdots a_{n-1}1\Box_1}(h_{n}); 
\textcircled{\scriptsize12}=X_{a_1\cdots a_{n-1}11}(h_{n}).$$
}
\end{center}
\end{figure}

Start with $h_n=f_0$ for every $n\in\{0,1,\ldots,N-1\}$.
Let $n\geq N$ and suppose $h_{n-1},h_{n-2}$ has been defined
so that the following holds:
\begin{itemize}
\item[(P$)_{n-1}$:] $h_{n-1}(X_{a_0\cdots a_{n-1}}(h_{n-1}))=X_{a_1\cdots a_{n-1}}(h_{n-2})$  for every $a_0\cdots a_{n-1}\in W_{n}$:
\end{itemize}
Put 
$$G_{n-1}=\bigcup_{k=1}^{n-1}\bigcup_{a_0\cdots a_{k-1}\in W_{k}}\bigcup_{i=1}^{p-1} X_{a_0\cdots a_{k-1}\Box_i}(h_{n-1}).$$
The open intervals $X_{a_0\cdots a_{k-1}\Box_i}(h_{n-1})$ in the union are called {\it a gap of of order $k$}.
The total number of gaps of order $k$ is $p^{k}(p-1)$.
Note that
$$X=G_{n-1}\bigcup\left(\bigcup_{a_0\cdots a_{n-1}\in W_{n}}
X_{a_0\cdots a_{n-1}}(h_{n-1})\right),$$ where all unions are disjoint.
Define $h_{n}$ so that the following holds  (See FIGURE 1):

\begin{itemize}
\item[(i)] $h_{n}\equiv h_{n-1}$ on
$G_{n-1}$;
\end{itemize}
For each $a_0\cdots a_{n-1}\in W_{n}$,
$h_{n}|_{X_{a_0\cdots a_{n-1}}(h_{n-1})}$ is defined as follows:
\begin{itemize}
\item[(ii)] for every $i\in\{0,1,\ldots,p-1\}$, $$h_{n}^{-1}(X_{a_1\cdots a_{n-1}i}(h_{n-1}))\cap
 X_{a_{0}\cdots a_{n-1}}(h_{n-1})=X_{a_{0}\cdots a_{n-1}i}(h_{n})$$
and
$$Dh_{n}|_{X_{a_{0}\cdots a_{n-1}i}(h_{n})}\equiv \pm E_{a_{0}\cdots a_{n-1}i}(\phi),$$ 
where the sign is $+$ if $Df_0>0$ on $[\alpha_{a_0},\beta_{a_0}]$
and $-$ otherwise;

\item[(iii)] for every $i\in\{1,\ldots,p-1\}$, $Dh_{n}|_{X_{a_{0}\cdots a_{n-1}\Box_i}(h_{n})}$ is a constant function.
\end{itemize}
Since $h_{n}$ is required to be continuous, there is no ambiguity in this definition.
Note that (P$)_{n}$ holds, which recovers the assumption of the induction.

There is a difference between the transition from $N-1$ to $N$ and that from $n-1$ to $n$, $n\geq N+1$.
Since $h_{N-1}=f_0$, an ``overhang'' may happen, in which case the definition of $h_{N}|_{X_{a_{0}\cdots a_{N-1}}(f_0)}$ does not make sense.
 As developed in the proof of Lemma \ref{d-ini} below.
the overhang does hot happen for an appropriately chosen $\phi$ and $N$.





\subsection{Analytic estimates on the sequence of continuous piecewise linear maps}\label{gapsp}
In this subsection we develop three estimates
on the sequence $\{h_n(\phi,N)\}_{n\geq0}$.
The next lemma states that
$\{h_n\}_{n\geq0}$ respects the proportions of gaps.
\begin{lemma}\label{d-ini}
There exist  a neighborhood $\hat{\mathscr{U}}$ of $\phi_0$ and $\hat N\geq1$ such that
the following holds
for every $\phi\in\hat{\mathscr{U}}$ and every integer $N>\hat N$:
the sequence $\{h_n(\phi,N)\}_{n\geq N}$ is well-defined, and
for every $i\in\{1,\ldots,p-1\}$,
every $n\geq 1$ and every $a_0\cdots a_{n-1}\in W_{n}$,
$$\Delta^{(i)}_{a_{0}\cdots a_{n-1}}(h_{n}(\phi,N))\leq 2K(f_0).$$
\end{lemma}
\begin{proof} 
We argue in two steps.
\medskip

\noindent{\it Step 1: \underline{Well-definedness of $\{h_n(\phi,N)\}_{n\geq N}$.}}
If $n\geq N+1$ then the transition from $n-1$ to $n$ makes sense.
It suffices to show that the transition from $N-1$ to $N$ makes sense.

Fix $\delta\in(0,1)$ such that
$\frac{(1+\delta)^2}{1+K(f_0)}<1$.
Choose a neighborhood $\hat{\mathscr{U}}$ of $\phi_0$ and $\hat N\geq1$
such that for every integer $N\geq \hat N$, every $\phi\in \hat{\mathscr{U}}$,
every $a_0\cdots a_{N-1}\in W_N$, every $i\in\{1,\ldots,p-1\}$ and every $j\in\{i-1,i\}$
the following holds:
\begin{equation}\label{get1}
\sup_{x\in    X_{a_{0}\cdots a_{N-1}j}(f_0)   }|Df_0(x)|\leq (1+\delta)E_{a_{0}\cdots a_{N-1}j}(\phi_0);\end{equation}
\begin{equation}\label{get2}
E_{a_{0}\cdots a_{N-1}j}(\phi_0)\leq(1+\delta)E_{a_{0}\cdots a_{N-1}j}(\phi).\end{equation}
The second condition follows from Lemma \ref{variation}.
Lemma \ref{distortion} implies
\begin{equation*}
\left|X_{a_{0}\cdots a_{N-1}\widehat i}(f_0)\right|^{-1}\sum_{j\in\{i-1,i\}}\left|X_{a_{0}\cdots a_{N-1}j}(f_0)\right|\leq\frac{1}{1+K(f_0)}.
\end{equation*}
By the Mean Value Theorem, for each $j\in\{i-1,i\}$ there exists $\xi_j\in  X_{a_{0}\cdots a_{N-1}j}(f_0)$
such that $|Df_0(\xi)|\left| X_{a_{0}\cdots a_{N-1}j}(f_0)\right|= \left|X_{a_{1}\cdots a_{N-1}j}(f_0)\right|$.
Hence
\begin{equation*}
\left|X_{a_{0}\cdots a_{N-1}\widehat i}(f_0)\right|^{-1}\sum_{j\in\{i-1,i\}}\frac{\left|X_{a_{1}\cdots a_{N-1}j}(f_0)\right|}{|Df_0(\xi_j)|}\leq
\frac{1}{1+K(f_0)}.
\end{equation*}
\eqref{get1} and \eqref{get2} yield 
\begin{equation*}
\left|X_{a_{0}\cdots a_{N-1}\widehat i}(f_0)\right|^{-1}\sum_{j\in\{i-1,i\}}\frac{\left|X_{a_{1}\cdots a_{N-1}j}(f_0)\right|}{E_{a_{0}\cdots a_{N-1}j}(\phi)}\leq
\frac{(1+\delta)^2}{1+K(f_0)}<1.
\end{equation*}
This condition prevents the overhang mentioned in the last paragraph of Sect.\ref{PL}.
Hence the transition from $N-1$ to $N$ makes sense.
\medskip

\noindent{\it Step 2: \underline{Proof of the inequality.}}
Let $N\geq \hat N$ be an integer and $\phi\in \hat{\mathscr{U}}$.
If $n\in\{0,1,\ldots,N-1\}$ then $h_n=f_0$, and 
 Lemma \ref{distortion} gives
$\Delta^{(i)}_{a_{0}\cdots a_{n-1}}(h_n)\leq K(f_0).$
Suppose $n=N$.
The definition of $h_{N}$ gives
\begin{equation}\label{deltE1}
\Delta^{(i)}_{a_{0}\cdots a_{N-1}}(h_{N})=\frac{\sum_{j\in\{i-1,i\}}\frac{|X_{a_1\cdots a_{N-1}j}(f_0)|}{E_{a_0\cdots
a_{N-1}j}(\phi)}}
{\int_{ X_{a_1\cdots a_{N-1}\widehat i}(f_0)                } |Df_0(f_0^{-1}(y))|^{-1}dy-
\sum_{j\in\{i-1,i\}}\frac{|X_{a_1\cdots a_{N-1}j}(f_0)|}{E_{a_0\cdots a_{N-1}j}(\phi)}}.
\end{equation}
To estimate the denominator of this fraction, put $$s=\sup_{ x\in X_{a_0\cdots a_{N-1}}(f_0)}|Df_0(x)|\quad\text{and}
\quad\eta=2V_{a_0\cdots a_{N-1}}(\phi_0)+M(e^{\phi_0}-e^\phi).$$
Shrinking $\hat{\mathscr{U}}$ and enlarging $\hat N$ if necessary,
we may assume
\begin{equation}\label{s}s>\max\{\eta K(f_0),\eta\}.\end{equation}
 We have $s^{-1}|X_{a_1\cdots a_{N-1}\widehat i}(f_0)|\leq\int_{X_{a_1\cdots a_{N-1}\widehat i}(f_0)} |Df_0(f_0^{-1}(y))|^{-1}dy$,
and Lemma \ref{variation} gives
$E_{a_0\cdots a_{N-1}j}(\phi)\geq s-\eta$ for every $j\in\{i-1,i\}$.
Hence, the denominator of the fraction in \eqref{deltE1} is bounded from below by
\begin{align*}
s^{-1}\left|X_{a_1\cdots a_{N-1}\widehat i}(f_0)\right|&-\frac{1}{s-\eta}\sum_{j\in\{i-1,i\}}|X_{a_1\cdots a_{N-1}j}(f_0)|\\
&=\frac{s\left|X_{a_1\cdots a_{N-1}\Box_i}(f_0)\right|-\eta \left|X_{a_1\cdots a_{N-1}\widehat i}(f_0)\right|}{s(s-\eta)}\\
&\geq\left|X_{a_1\cdots a_{N-1}\Box_i}(f_0)\right|\cdot\frac{s-\eta K(f_0)}{s(s-\eta)}>0,\end{align*}
where the first inequality follows from Lemma \ref{distortion} and
the last from \eqref{s}.
From \eqref{deltE1} we have
\begin{equation}\label{deltE2}\Delta^{(i)}_{a_{0}\cdots a_{N-1}}(h_{N})\leq
\frac{\sum_{j\in\{i-1,i\}}\frac{|X_{a_1\cdots a_{N-1}j}(f_0)|}{E_{a_0\cdots a_{N-1}j}(\phi)}}
{s^{-1}|X_{a_1\cdots a_{N-1}\widehat i}(f_0)|-\sum_{j\in\{i-1,i\}}
\frac{|X_{a_1\cdots a_{N-1}j}(f_0)|}{E_{a_0\cdots a_{N-1}j}(\phi)}}.\end{equation}
From Lemma \ref{variation} and the definition of $\eta$, for every $j\in\{i-1,i\}$ we have
\begin{align*} s-E_{a_0\cdots a_{N-1}j}(\phi)
&\leq  F_{a_0\cdots a_{N-1}}(\phi_0)-E_{a_0\cdots a_{N-1}}(\phi)\\
&=V_{a_0\cdots a_{N-1}}(\phi_0)
+E_{a_0\cdots a_{N-1}}(\phi_0)-E_{a_0\cdots a_{N-1}}(\phi)\\
&\leq\eta.\end{align*}
Substituting $E_{a_0\cdots a_{N-1}j}(\phi)\geq s-\eta$ 
into \eqref{deltE2} yields
\begin{align*}\Delta^{(i)}_{a_0\cdots a_{N-1}}(h_{N})&\leq\frac{\sum_{j\in\{i-1,i\}}|X_{a_1\cdots a_{N-1}j}(f_{0})|}
{|X_{a_1\cdots a_{N-1}\Box_i}(f_0)|-(\eta/s)|X_{a_1\cdots a_{N-1}\widehat i}(f_0)|}\\
&=\left(1-\frac{\eta}{s}\frac{|X_{a_1\cdots a_{N-1}\widehat i}(f_0)|}{|X_{a_1\cdots a_{N-1}\Box_i}(f_0)|}\right)^{-1} 
\frac{\sum_{j\in\{i-1,i\}}|X_{a_1\cdots a_{N-1}j}(f_{0})|}
{|X_{a_1\cdots a_{N-1}\Box_i}(f_0)|}\\
&\leq \left(1-\frac{\eta}{s}(1+K(f_0))\right)^{-1}\Delta^{(i)}_{a_1\cdots a_{N-1}}(f_{0})\\
&\leq 2\Delta^{(i)}_{a_1\cdots a_{N-1}}(f_{0})\leq 2K(f_0).
\end{align*}

It is left to treat the case $n\geq N+1$.
The construction of $h_n$ from $h_{n-1}$ in Sect.\ref{PL} implies 
$$\Delta^{(i)}_{a_0\cdots a_{n-1}}(h_{n})\leq\Delta^{(i)}_{a_1\cdots a_{n-1}}(h_{n-1}).$$
Using this inductively yields
$\Delta^{(i)}_{a_0\cdots a_{n-1}}(h_{n})\leq\Delta^{(i)}_{a_{n-N}\cdots a_{n-1}}(h_{N})\leq 2K(f_0).$
\end{proof}

In what follows, let $\hat{\mathscr{U}}$ be the neighborhood of $\phi_0$ in $C(\Sigma_p)$ and $\hat N\geq1$ 
the number in the statement of Lemma \ref{d-ini}.
For an integer $n\geq1$ and
 $a_0\cdots a_{n-1}\in W_{n}$ 
define $\tau_{a_0\cdots a_{n-1}\Box_i}$ to be the constant value of $|Dh_{n}|$ on 
$X_{a_0\cdots a_{n-1}\Box_i}(h_{n})$.
\begin{lemma}\label{slope}
The following holds
for every $\phi\in\hat{\mathscr{U}}$, every integer $N>\hat{N}$ and $\{h_n(\phi,N)\}_{n\geq0}$:
\begin{itemize}
\item[(a)] For every $i\in\{1,\ldots,p-1\}$,
every $n\geq N+1$ and every $a_0\cdots a_{n-1}\in W_{n}$,
$$\tau_{a_0\cdots a_{n-1}\Box_i}>E_{a_0\cdots a_{n-1}}(\phi)-2M(e^{-2\phi})
V_{a_0\cdots a_{n-1}}(\phi)E_{a_0\cdots a_{n-1}}^2(\phi)K(f_0).$$
Moreover, for every $i\in\{1,\ldots,p-1\}$ and every $a_0\cdots a_{N-1}\in W_{N}$,
\begin{align*}
\tau_{a_0\cdots a_{N-1}\Box_i}>&E_{a_0\cdots a_{N-1}}(\phi_0)-2M(e^{-\phi-\phi_{0}})
\left(2V_{a_0\cdots a_{N-1}}(\phi_0)+M(e^{\phi}-e^{\phi_0})\right)\\
&\times E_{a_0\cdots a_{N-1}}^2(\phi_0)
K(f_0).\end{align*}
\item[(b)]\label{C1} For every integers $m$, $n$ with $m> n\geq N$, every 
$a_0\cdots a_{m-1}\in W_{m}$ 
and every $i$, $j\in\{1,2,\ldots,p-1\}$,
$$|\tau_{a_0\cdots a_{m-1}\Box_j}-\tau_{a_0\cdots a_{n-1}\Box_i}|
\leq  (1+M(e^{2\phi})K(f_0))V_{a_0\cdots a_{n-1}}(\phi).$$
\end{itemize}
\end{lemma}
\begin{proof}
As for (a),  let $i\in\{1,\ldots,p-1\}$
and $a_0\cdots a_{n-1}\in W_{n}$.
We first consider the case $n\geq N+1$.
From the definition of $h_{n}$,
\begin{align*}
\tau_{a_0\cdots a_{n-1}\Box_i}^{-1}=&|X_{a_1\cdots a_{n-1}\Box_i}(h_{n-1})|^{-1}\left(\frac{|X_{a_1\cdots a_{n-1}\widehat i}(h_{n-1})|}
{E_{a_0\cdots a_{n-1}}(\phi)}-\sum_{j\in\{i-1,i\}}\frac{|X_{a_1\cdots a_{n-1} j}(h_{n-1})|}
{E_{a_0\cdots a_{n-1} j}(\phi)}\right)\\
&=E_{a_{0}\cdots a_{n-1}}^{-1}(\phi)+
\frac{\sum_{j\in\{i-1,i\}}\frac{E_{a_{0}\cdots a_{n-1} j}(\phi)-
E_{a_{0}\cdots a_{n-1}}(\phi)}{E_{a_0\cdots a_{n-1}}(\phi)
E_{a_{0}\cdots a_{n-1} j}(\phi)}|X_{a_1\cdots a_{n-1} j}(h_{n-1})|}{|X_{a_1\cdots a_{n-1}\Box_i}(h_{n-1})|}.
\end{align*}
Hence
\begin{align*}\tau_{a_0\cdots a_{n-1}\Box_i}^{-1}-E_{a_0\cdots a_{n-1}}^{-1}(\phi)&=
\frac{\sum_{j\in\{i-1,i\}}
\frac{E_{a_{0}\cdots a_{n-1}j}(\phi)-E_{a_{0}\cdots a_{n-1}}(\phi)}
{E_{a_0\cdots a_{n-1}}(\phi)E_{a_{0}\cdots a_{n-1}j}(\phi)}|X_{a_1\cdots a_{n-1} j}(h_{n-1})|}
{|X_{a_1\cdots a_{n-1}\Box_i}(h_{n-1})|}\\
&\leq
M(e^{-2\phi})\frac{\sum_{j\in\{i-1,i\}}
V_{a_{0}\cdots a_{n-1}}(\phi)|X_{a_1\cdots a_{n-1}j}(h_{n-1})|}
{|X_{a_1\cdots a_{n-1}\Box_i}(h_{n-1})|}
\\
&=  M(e^{-2\phi})V_{a_0\cdots a_{n-1}}(\phi)\Delta^{(i)}_{a_1\cdots a_{n-1}}(h_{n-1}).
\end{align*}
This yields
\begin{align*}
\tau_{a_0\cdots a_{n-1}\Box_i}& \geq
 E_{a_0\cdots a_{n-1}}(\phi)
-\frac{V_{a_0\cdots a_{n-1}}(\phi)E_{a_0\cdots a_{n-1}}^2(\phi)\Delta^{(i)}_{a_1\cdots a_{n-1}}(h_{n-1})}
{V_{a_0\cdots a_{n-1}}(\phi)E_{a_0\cdots a_{n-1}}(\phi)
\Delta^{(i)}_{a_1\cdots a_{n-1}}(h_{n-1})+1/M(e^{-2\phi})}\\
&> E_{a_0\cdots a_{n-1}}(\phi)-M(e^{-2\phi})V_{a_0\cdots a_{n-1}}(\phi)E_{a_0\cdots a_{n-1}}^2(\phi)\Delta^{(i)}_{a_1\cdots a_{n-1}}(h_{n-1})\\
&\geq  E_{a_0\cdots a_{n-1}}(\phi)-2M(e^{-2\phi})V_{a_0\cdots a_{n-1}}(\phi)
E_{a_0\cdots a_{n-1}}^2(\phi)K(f_0).
\end{align*}
The last inequality follows from Lemma \ref{d-ini}.

A proof for the case $n=N$ is analogous to the above argument modulo minor differences.
We simply replace $E_{a_0\cdots a_{N-1}}(\phi)$ by $\sup_{ x\in X_{a_0\cdots a_{N-1}}(f_0)}|Df_0(x)|$ and argue in the same way.
  On the fraction in the summand,  for every $j\in\{i-1,i\}$,
 \begin{align*}E_{a_0\cdots a_{N-1}j}(\phi)-\sup_{ x\in X_{a_0\cdots a_{N-1}}(f_0)}|Df_0(x)|&\leq 
 E_{a_0\cdots a_{N-1}j}(\phi)-E_{a_0\cdots a_{N-1}}(\phi_0)\\
 &\leq  F_{a_0\cdots a_{N-1}}(\phi)-E_{a_0\cdots a_{N-1}}(\phi_0)\\
 &\leq2V_{a_0\cdots a_{N-1}}(\phi_0)+M(e^{\phi}-e^{\phi_{0}}),\end{align*}
 where the last inequality follows from the second estimate in Lemma \ref{variation}.
 This completes the proof of Lemma \ref{slope}(a).


From the construction in Sect.\ref{PL}, 
$$\max
\{                 \tau_{a_0\cdots a_{m-1}\Box_j},\tau_{a_0\cdots a_{n-1}\Box_i}         \}
\leq F_{a_0\cdots a_{n-1}}(\phi).$$
Lemma \ref{slope}(a) implies
\begin{align*}
\tau_{a_0\cdots a_{m-1}\Box_j}&>E_{a_0\cdots a_{m-1}}(\phi)-V_{a_0\cdots a_{m-1}}(\phi)
E_{a_0\cdots a_{m-1}}^2(\phi)K(f_0)
\end{align*}
and
\begin{align*}
\tau_{a_0\cdots a_{n-1}\Box_i}&>E_{a_0\cdots a_{n-1}}(\phi)-V_{a_0\cdots a_{n-1}}(\phi)
E_{a_0\cdots a_{n-1}}^2(\phi)K(f_0).
\end{align*}
If 
$\tau_{a_0\cdots a_{n-1}\Box_i}>\tau_{a_0\cdots a_{m-1}\Box_j}$ then
\begin{align*}
\tau_{a_0\cdots a_{n-1}\Box_i}-\tau_{a_0\cdots a_{m-1}\Box_j}&\leq \left(1+E_{a_0\cdots a_{n-1}}^2(\phi)K(f_0)\right)V_{a_0\cdots a_{n-1}}(\phi).\end{align*}
In the case $\tau_{a_0\cdots a_{n-1}\Box_i}<\tau_{a_0\cdots a_{m-1}\Box_j}$
we get the same inequality.
This completes the proof of Lemma \ref{slope}(b).
\end{proof}

\subsection{Perturbation to $C^1$ maps}\label{perturb}
We have constructed a sequence $\{h_n(\phi,N)\}_{n\geq0}$ of continuous piecewise linear maps.
For each $n\geq N$ we define a $C^1$ map $f_n\colon X\to[0,1]$ 
by perturbing $h_n$ on each gap of order $k=N,N+1,\ldots,n$.

Start with $f_n=h_n$ on $G_N$.
Let $k\geq N$ and $a_0\cdots a_{k-1}\in W_k$.
Define $f_n|_{X_{a_0\cdots a_{k-1}}(h_{n-1})}$ as follows.
Recall that 
$$X_{a_0\cdots a_{k-1}}(h_{n-1})=\bigcup_{i=1}^{p-1}X_{a_0\cdots a_{k-1}\widehat i}(h_{n-1}),$$
and for every $i\in\{1,2,\ldots,p-1\}$,
$$X_{a_0\cdots a_{k-1}\widehat i}(h_{n-1})=X_{a_0\cdots a_{k-1}(i-1)}(h_{n-1})\bigcup
X_{a_0\cdots a_{k-1}\Box_i}(h_{n-1})\bigcup X_{a_0\cdots a_{k-1}i}(h_{n-1}).$$
For every $x\in  X_{a_0\cdots a_{k-1}(i-1)}(h_{n-1})\bigcup X_{a_0\cdots a_{k-1}i}(h_{n-1})$,
set $f_n(x)=h_n(x)$.
In order to define $f_n$ on $X_{a_0\cdots a_{k-1}\Box_i}(h_{n-1})$
we need the next lemma.
 
\begin{lemma}\label{smoothing}
Let $E>0$, $E'>0$, $\tau>0$ be such that $2\tau-(E+E')/2>0$
and $I=[\alpha,\beta]$ a compact interval.
There exists a $C^1$ diffeomorphism $\varphi\colon I\to\mathbb R$ such that $D\varphi(\alpha)=E$, $D\varphi(\beta)=E'$,
$|\varphi(I)|=\tau|I|$ and for every $x\in I$,
$$|D\varphi(x)-\tau|\leq\max\{E,E'\}-\tau.$$
\end{lemma}
\begin{proof}
Define a continuous function $g\colon I\to\mathbb R$ by the following set of conditions:
$g(\alpha)=E$, $g(\beta)=E'$, $g((\alpha+\beta)/2)=2\tau-(E+E')/2$, $g$ is linear on $[\alpha,(\alpha+\beta)/2]$ and 
$[(\alpha+\beta/2),\beta]$. 
Define $\varphi\colon I\to\mathbb R$ by $\varphi(x)=\int_0^{x} g(y)dy.$
It is easy to check that $f$ satisfies the desired properties apart from the last one.
To show the last property, note that
$2\tau-(E+E')/2\leq D\varphi(x)\leq\max\{E,E'\}$ for every $x\in I$.
Let $x\in I$ and suppose $D\varphi(x)-\tau<0$.
Then $$|D\varphi(x)-\tau|=-D\varphi(x)+\tau\leq (E+E')/2-\tau\leq\max\{E,E'\}-\tau.$$
If $D\varphi(x)-\tau\geq0$, then
 $$|D\varphi(x)-\tau|=D\varphi(x)-\tau\leq\max\{E,E'\}-\tau.$$
 This completes the proof of the lemma.
\end{proof}
 Let $\psi\colon \overline{X_{a_0\cdots a_{k-1}\Box_i}(h_{n-1})}\to\mathbb R$ be
a $C^1$ diffeomorphism for which the conclusion of Lemma \ref{smoothing} holds
with  $E=E_{a_0\cdots a_{k-1}(i-1)}(\phi)$, $E'=E_{a_0\cdots a_{k-1}i}(\phi)$,
 $\tau=\tau_{a_0\cdots a_{k-1}\Box_i}$, $I=\overline{X_{a_0\cdots a_{k-1}\Box_i}(h_{n-1})}$.
 Lemma \ref{slope} ensures the condition
 $2\tau-(E+E')/2>0$ in Lemma \ref{smoothing}.
  Define 
$$f_n(x)=h_n(l)\pm \psi(x)\text{ for every $x\in  X_{a_0\cdots a_{k-1}\Box_i}(h_{n-1})$},$$
where $l$ denotes the left boundary point of $X_{a_0\cdots a_{k-1}\Box_i}(h_{n-1})$.
The sign is $+$ if $Df_0>0$ on $(r,s)$ and $-$ otherwise.
This finishes the definition of $f_n$.
 Lemma \ref{smoothing} implies that $f_n\colon X\to[0,1]$ is $C^1$ and satisfies (E1).

\begin{lemma}\label{gapd}
The following holds
for every $\phi\in\hat{\mathscr{U}}$, every integer $N>\hat{N}$  and the sequences $\{h_n(\phi,N)\}_{n\geq0}$ and $\{f_n(\phi,N)\}_{n\geq0}$:
Let $n\geq N+1$. For every $k\in\{N,N+1,\ldots n\}$, every $a_0\cdots a_{k-1}\in W_k$,
every $i\in\{1,2,\ldots,p-1\}$ and every $x\in X_{a_0\cdots a_{k-1}\widehat i}(h_{n-1})$ we have
\begin{align*}||Df_n(x)|-\tau_{a_0\cdots a_{k-1}\Box_i}|\leq M(e^{-2\phi})V_{a_0\cdots a_{k-1}}(\phi)
E_{a_0\cdots a_{k-1}}^2(\phi)K(f_0).\end{align*}
Moreover, for every $a_0\cdots a_{N-1}\in W_N$,
every $i\in\{1,2,\ldots,p-1\}$ and every $x\in X_{a_0\cdots a_{N-1}\widehat i}(f_{0})$ we have
\begin{align*}||Df_N(x)|-\tau_{a_0\cdots a_{N-1}\Box_i}|\leq2M(e^{-\phi-\phi_0})V_{a_0\cdots a_{N-1}}(\phi_0)E_{a_0\cdots a_{N-1}}^2(\phi_0)K(f_0).\end{align*}

\end{lemma}
\begin{proof}
From Lemma \ref{smoothing} and Lemma \ref{slope}(a) we have
\begin{align*}||Df_n(x)|-\tau_{a_0\cdots a_{k-1}\Box_i}|&\leq \max\{E_{a_0\cdots a_{k-1}(i-1)}(\phi),
E_{a_0\cdots a_{k-1}i}(\phi)\}-\tau_{a_0\cdots a_{k-1}\Box_i}\\
&\leq E_{a_0\cdots a_{k-1}}(\phi)-\tau_{a_0\cdots a_{k-1}\Box_i}\\
&\leq M(e^{-2\phi})V_{a_0\cdots a_{k-1}}(\phi)E_{a_0\cdots a_{k-1}}^2(\phi)K(f_0).\end{align*}
A proof of the second inequality in the lemma is analogous and hence we omit it.
\end{proof}

\subsection{Cauchy property}\label{cauchyp}
Starting from a $C^2$ map $f_0\in\mathscr{E}$ we have constructed a sequence $\{f_n(\phi,N)\}_{n\geq0}$ of $C^1$ maps on $X$.
We show that $\{f_n\}_{n\geq0}$ is a Cauchy sequence in $\mathscr{E}$
which is contained in a $C^1$ neighborhood of $f_0$.
\begin{lemma}\label{C1start}
 For every $\varepsilon>0$ there exist a neighborhood $\mathscr{U}$ of $\phi_0=\log|Df_0|\circ\pi_{f_0}$ in $C(\Sigma_p)$ 
 and $N_0\geq1$ such that the following holds
 for every $\phi\in\mathscr{U}$, every integer $N>N_0$ and the sequence $\{f_n(\phi,N)\}_{n\geq0}$:
 \begin{itemize}  
\item[(a)] for every $n\geq N$, 
$$\|f_0-f_n\|_{C^1}\leq\varepsilon;$$

 \item[(b)]  $\{f_n\}_{n\geq0}$ is a Cauchy sequence in $\mathscr{E}$ with respect to the norm $\|\cdot\|_{C^1}$.

 \end{itemize}
 
  \end{lemma}
 \begin{proof}
Let $\varepsilon>0$.
 Let $\phi\in\hat{\mathscr{U}}$ and $N>\hat N$ be an integer.
Depending on $\varepsilon$
we will choose $\phi$ that is sufficiently close to $\phi_0$, and then choose a sufficiently large $N$.

 We first estimate $\|f_0-f_N\|_{C^1}$.
 Let $a_0\cdots a_{N-1}\in W_N$ and
 $i\in\{1,2,\ldots,p-1\}$. 
 From the construction,
$|f_0(x)-f_N(x)|\leq |X_{a_1\cdots a_{N-1}\widehat i}(f_0)|$
holds for every $x\in X_{a_0\cdots a_{N-1}\widehat i}(f_0)$.
(E2) for $f_0$ implies $|X_{a_1\cdots a_{N-1}\widehat i}(f_0)|\leq c\lambda^{-N+1}$.
Hence
\begin{equation}\label{fn1}\sup_{x\in X}|f_0(x)-f_N(x)|\leq c\lambda^{-N+1}\leq\frac{\varepsilon}{4},\end{equation}
where the last inequality holds for sufficiently large $N$.

 Let $a_0\cdots a_{N-1}\in W_N$ and
 $i\in\{1,2,\ldots,p-1\}$. For every $x\in X_{a_0\cdots a_{N-1}\widehat i}(f_0)$
we have
\begin{equation*}
|Df_0(x)-Df_N(x)|\leq||Df_0(x)|-\tau_{a_0\cdots a_{N-1}\Box_i}|+|\tau_{a_0\cdots a_{N-1}\Box_i}-|Df_N(x)||.\end{equation*}
The second term is bounded by Lemma \ref{gapd}.
We estimate the first term.
If $|Df_0(x)|\leq\tau_{a_0\cdots a_{N-1}\Box_i}$, then from $\tau_{a_0\cdots a_{N-1}\Box_i}\leq E_{a_0\cdots a_{N-1}}(\phi)$ in Lemma \ref{slope}(a) and from Lemma \ref{variation} we have
\begin{align*}
||Df_0(x)|-\tau_{a_0\cdots a_{N-1}\Box_i}|&\leq  E_{a_0\cdots a_{N-1}}(\phi)-E_{a_0\cdots a_{N-1}}(\phi_0)\\
&\leq V_{a_0\cdots a_{N-1}}(\phi)+M(e^\phi-e^{\phi_0}).\end{align*}
If $|Df_0(x)|>\tau_{a_0\cdots a_{N-1}\Box_i}$, then from
$|Df_0(x)|\leq F_{a_0\cdots a_{N-1}}(\phi_0)$ and the second inequality in Lemma \ref{slope}(a) we have
\begin{align*}||Df_0(x)|-\tau_{a_0\cdots a_{N-1}\Box_i}|\leq& V_{a_0\cdots a_{N-1}}(\phi_0)
+2M(e^{-\phi-\phi_{0}})V_{a_0\cdots a_{N-1}}(\phi_0)\\
&+M(e^{-\phi-\phi_{0}})M(e^{\phi_0}-e^\phi)E_{a_0\cdots a_{N-1}}^2(\phi_0)
K(f_0).\end{align*}
It follows that
\begin{equation}\label{fn2}\sup_{x\in X}|Df_0(x)-Df_N(x)|\leq \frac{\varepsilon}{4},\end{equation}
provided $\phi$ is sufficiently close to $\phi_0$ and $N$ is sufficiently large.
From \eqref{fn1} and \eqref{fn2} we obtain
\begin{equation}\label{fn}\|f_0-f_N\|_{C^1}\leq\frac{\varepsilon}{2}.\end{equation}

 
 Let $\varepsilon_0$, $c_0$, $\lambda_0$ be the constants
for which the conclusions of Lemma \ref{standard} holds with respect to $f_0$.
Let $\varepsilon\in(0,\varepsilon_0)$.
Let $m$, $n$ be integers with $m>n\geq N$. 
We estimate $\|f_n-f_m\|_{C^1}$.
If $x\in X$ is contained in a gap of order $\leq n$ we have
$h_m(x) =h_n(x)$, and thus
 $f_m(x)=f_n(x)$.
 Suppose $x\in X$ is not contained in any gap of order $\leq n$.
Then, there exist $a_0\cdots a_{n-1}\in W_n$ and $i\in\{1,2,\ldots,p-1\}$ such that $x\in X_{a_0\cdots a_{n-1}\widehat i}(h_{n-1})$.
The construction of $\{h_n\}_{n\geq0}$ in Sect.\ref{PL} implies 
 \begin{align*}\label{trans-eq3}
  |f_n(x)-f_m(x)|&\leq\sup\{|h_n(y)-h_m(y)|\colon y\in X_{a_0\cdots a_{n-1}\widehat i}(h_{n-1})\}\\
  &=|X_{a_1\cdots a_{n-1}\widehat i}(h_{n-1})|\\
  &=|X_{a_1\cdots a_{n-1}\widehat i}(f_{n-1})|.
 \end{align*}
 Since $x\in X_{a_0\cdots a_{n-1}\widehat i}(h_{n-1})$ and $a_0\cdots a_{n-1}\in W_n$ are arbitrary,
we obtain $$ \sup_{x\in X}|f_n(x)-f_m(x)|\leq\sup_{\stackrel{a_1\cdots a_{n-1}\in W_{n-1}}{i\in\{1,2,\ldots,p-1\}}}
|X_{a_1\cdots a_{n-1}\widehat i}(f_{n-1})|.$$

Proceeding to the estimate of derivatives, again
let $x\in X$ and let $a_0\cdots a_{n-1}\in W_n$, $i\in\{1,2,\ldots,p-1\}$ be such that $x\in X_{a_0\cdots a_{n-1}\widehat i}(h_{n-1})$.
 We treat two cases separately.
 \medskip

\noindent{\it Case I. \underline{$x$ is not contained in a gap of order $\leq m$.}}
We have $|Df_n(x)|=E_{a_0\cdots a_{n-1}(i-1)}(\phi)$ or $|Df_n(x)|=E_{a_0\cdots a_{n-1}i}(\phi)$, and
$|Df_m(x)|=E_{a_0\cdots a_{m-1}}(\phi)$. Hence \begin{equation*}\label{c1eq1}
 |Df_n(x)-Df_m(x)|\leq V_{a_0\cdots a_{n-1}}(\phi).
 \end{equation*}

\noindent{\it Case II. \underline{$x$ is contained in a gap of order $k\in[n,m]$.}}
Let $a_0\cdots a_{k-1}\in W_k$ and $j\in\{1,2,\ldots,p-1\}$ be such that $x\in X_{a_0\cdots a_{k-1}\Box_j}(h_{m-1})$.
 We have
 \begin{align*}
 |Df_n(x)-Df_m(x)|\leq&||Df_n(x)|-\tau_{a_0\cdots a_{n-1}\Box_i}|+|\tau_{a_0\cdots a_{n-1}\Box_i}-\tau_{a_0\cdots a_{k-1}\Box_j}|\\
 &+|\tau_{a_0\cdots a_{k-1}\Box_j}-|Df_m(x)||.
 \end{align*}
The first and the third terms are bounded by Lemma \ref{gapd}.
For the second term, Lemma \ref{slope}(b) gives
$$|\tau_{a_0\cdots a_{n-1}\Box_i}-\tau_{a_0\cdots a_{k-1}\Box_j}|\leq \left(1+M(e^{2\phi})K(f_0)\right)
V_{a_0\cdots a_{n-1}}(\phi).$$
Hence we obtain
\begin{align*}
|Df_n(x)-Df_m(x)|\leq&2M(e^{-2\phi})\sum_{\ell\in\{k,n\}}V_{a_0\cdots a_{\ell-1}}(\phi)
E_{a_0\cdots a_{\ell-1}}^2(\phi)K(f_0)\\
&+\left(1+M(e^{2\phi})K(f_0)\right)V_{a_0\cdots a_{n-1}}(\phi).
 \end{align*}
 Since $x\in X_{a_0\cdots a_{n-1}\widehat i}(h_{n-1})$ and $a_0\cdots a_{n-1}\in W_n$ are arbitrary,
we obtain
 $$\sup_{x\in X}|Df_n(x)-Df_m(x)|\leq {\rm const}\cdot \sup_{a_0\cdots a_{n-1}\in W_n}V_{a_0\cdots a_{n-1}}(\phi).$$
 where the multiplicative constant only depends on $\phi$ and $f_0$.
 Overall, 
for every integers $m$, $n$ with $m>n\geq N$,
\begin{align}\label{C1cau} \|f_n-f_m\|_{C^1}\leq&
\sup_{\stackrel{a_1\cdots a_{n-1}\in W_{n-1}}{i\in\{1,2,\ldots,p-1\}}}|X_{a_1\cdots a_{n-1}\widehat i}(f_{n-1})|\\
&+{\rm const}\cdot \sup_{a_0\cdots a_{n-1}\in W_n}V_{a_0\cdots a_{n-1}}(\phi).\notag
\end{align}
Since $f_{N-1}=f_0$, for every $n\geq N+1$ we have
\begin{align*} \|f_N-f_n\|_{C^1}\leq&
\sup_{\stackrel{a_1\cdots a_{N-1}\in W_{N-1}}{i\in\{1,2,\ldots,p-1\}}}|X_{a_1\cdots a_{N-1}\widehat i}(f_{0})|\\
&+{\rm const}\cdot \sup_{a_0\cdots a_{N-1}\in W_N}V_{a_0\cdots a_{N-1}}(\phi)\leq\frac{\varepsilon}{2},\end{align*}
where the last inequality holds provided $\phi$ is sufficiently close to $\phi_0$ and $N$ is sufficiently large
depending on $\phi$. 
From this and \eqref{fn} we obtain $\|f_0-f_n\|_{C^1}\leq\varepsilon$
for every $n\geq N$.
Lemma \ref{standard}
gives $|X_{a_1\cdots a_{n-1}\widehat i}(f_{n-1})|\leq c_0^{-1}\lambda_0^{-n}$, and so
the first term of \eqref{C1cau} converges to zero as $n\to\infty$.
The convergence of the second term follows from
the uniform continuity of $\phi$.
\eqref{C1cau} implies that $\{f_n\}_{n\geq0}$ is a Cauchy sequence.
\end{proof}

\subsection{End of the proof of the Realization Lemma}\label{end}
We complete the proof of the Realization Lemma.

\begin{proof}[Proof of the Realization Lemma]

 Let $f_\infty$ denote the limit of the Cauchy sequence $\{ f_n\}_{n\geq 0}$.
Then $f_\infty$ satisfies (E1).
 By Lemma \ref{C1start},
 $\|f_0-f_\infty\|_{C^1}\leq\varepsilon$ holds. Lemma \ref{standard} implies that $f_\infty$
 satisfies (E2).
If $x\in\Lambda(f_\infty)$ and $\pi_{f_\infty}^{-1}(x)=\{a_n\}_{n\geq0}$
then
$x\in \bigcap_{n\geq 0} X_{a_0\cdots a_{n}}(f_\infty)$.
By construction, $|Df_n(x)|\equiv  E_{a_0\cdots a_{n}}(\phi)$ holds for every $n\geq N$, and
therefore $$\log|Df_\infty(x)|=\lim_{n\to\infty}\log|D f_n(x)|=\lim_{n\to\infty}\log E_{a_0\cdots a_{n}}(\phi)=\phi(\pi_{f_\infty}^{-1}(x)),$$
which yields $\log|Df_\infty|\circ\pi_{f_\infty}=\phi$.
\end{proof}

\section{On the proof of Theorem B}
In this last section we prove Theorem B.
In Sect.\ref{contreras} we recall the result of Contreras \cite{Con16}
on ergodic optimization for expanding maps.
In Sect.\ref{approx} we show that any map in $\mathscr{E}$ is approximated by another whose derivative
is locally constant.
In Sect.\ref{LIP} we refine the construction in Sect.\ref{PL} and prove a Lipschitz version of the Realization Lemma.
In Sect.\ref{endB} we complete the proof of Theorem B.

\subsection{The result of Contreras on optimization by periodic measures}\label{contreras}
Let $(Y,d)$ be a compact metric space and $\phi$ a real-valued Lipschitz continuous function on $Y$.
The Lipschitz norm of $\phi$ is given by
$$\|\phi\|_{\rm Lip}=\max_{y\in Y}|\phi(y)|+\sup_{\stackrel{y,z\in Y}{y\neq z}}\frac{|\phi(y)-\phi(z)|}{d(y,z)}.$$
Let ${\rm Lip}(Y,d)$ denote the space of Lipschitz continuous functions on $Y$
endowed with the topology given by the norm $\|\cdot\|_{\rm Lip}$.
A Lipschitz continuous map $T\colon Y\circlearrowleft$ is {\it expanding}  
if there exist an integer $m\geq1$ and a constant $\lambda>1$ such that for every $y\in Y$
there exist $\ell_y\in\{1,2,\ldots, m\}$, a neighborhood $U_y$ of $y$ in $Y$ and continuous maps $S_i\colon U_y\to Y$, $i\in\{1,\ldots,\ell_y\}$ such that the following holds:

\begin{itemize}
\item[-] $S_i(U_y)\cap S_j(U_y)=\emptyset$ if $i\neq j$; 

\item[-] $T^{-1}(U_y)=\bigcup_{i=1}^{\ell_y}S_i(U_y)$;

\item[-] $T( S_i(z))=z$ for every $i\in\{1,2,\ldots,\ell_y\}$ and every $z\in U_y$;

\item[-] $ d(z,w)\geq\lambda d(S_i(z),S_i(w))$ for every $i\in\{1,2,\ldots,\ell_y\}$ and every $z,w\in U_y$.
\end{itemize}

\begin{theoremcont}{\rm (\cite[Theorem A]{Con16})}
Let $(Y,d)$ be a compact metric space and $T\colon Y\circlearrowleft$ an expanding map.
There exists an open dense subset $\mathscr{O}_{(Y,d)}$ of ${\rm Lip}(Y,d)$ 
such that for every $\phi\in \mathscr{O}_{(Y,d)}$ there exists a unique $\phi$-minimizing measure, and it is 
supported on a periodic orbit.
\end{theoremcont}
We will apply Theorem 2 to the left shift $\sigma\colon\Sigma_p\circlearrowleft$.
The topology of $\Sigma_p$ is equivalent to the topology generated by the distance $d_\theta$
given by
$$d_\theta(\underline{a},\underline{b})=\begin{cases}0& \text{if $\underline{a}=\underline{b}$;}\\
\theta^{s(\underline{a},\underline{b})}& \text{otherwise,}\end{cases}$$
where $\theta\in(0,1)$, $\underline{a}=\{a_n\}_{n\geq0}$,  $\underline{b}=\{b_n\}_{n\geq0}$ and
$s(\underline{a},\underline{b})=\min\{n\geq0\colon a_n\neq b_n\}$.
Note that $\sigma$ is expanding in the above sense.

\subsection{Approximation by maps with locally constant derivatives}\label{approx}

For $f\in\mathscr{E}$ we introduce the following additional condition:
\begin{itemize}
\item[(E3)] there exists an integer $n\geq1$ such that the value of $Df$ is constant
on each connected component of $\bigcap_{k=0}^{n-1} f^{-k}(X)$.
\end{itemize}

\begin{lemma}\label{lc}
Let $f_0\in\mathscr{E}$. For any $\varepsilon>0$
there exists $f\in\mathscr{E}$ which satisfies
$\|f_0-f\|_{C^1}\leq\varepsilon$ and (E3).
\end{lemma}
\begin{proof}
Let $f_0\in\mathscr{E}$ and $\varepsilon>0$.
From (E2)  there exists an integer $n_0\geq1$ such that
for every $n\geq n_0$ and every $a_0\cdots a_{n-1}\in W_n$,
\begin{equation}\label{lipeq1}\left|X_{a_0\cdots a_{n-1}}(f_0)\right|\leq\frac{\varepsilon}{2}.\end{equation}
Since $Df_0$ is uniformly continuous, there exists an integer $n_1\geq1$ such that
for every $n\geq n_1$ and every $a_0\cdots a_{n-1}\in W_n$,
\begin{equation}\label{lipeq2}\max_{x,y\in X_{a_0\cdots a_{n-1}}(f_0)} |Df_0(x)-Df_0(y)|\leq\frac{\varepsilon}{4}.\end{equation}
Put $n_2=\max\{n_0,n_1\}$. Let $n\geq n_2$. 
Define $f\in\mathscr{E}$ so that the following holds:

\begin{itemize}
\item[-] (On the complement of the union of gaps of order $\leq n-1$)
Let $a_0\cdots a_{n-1}\in W_n$. Then
$f\left(X_{a_0\cdots a_{n-1}}(f_0)\right)=X_{a_1\cdots a_{n-1}}(f_0)$ and
for every $x\in X_{a_0\cdots a_{n-1}}(f_0)$,
$$Df(x)=\pm\frac{\left|X_{a_1\cdots a_{n-1}}(f_0)\right|}{\left|X_{a_0\cdots a_{n-1}}(f_0)\right|}.$$
The sign is $+$ if $Df_0(x)>0$ and $-$ otherwise;

\item[-] (On gaps of low order) Let $k\in\{1,\ldots,n_2-1\}$, $i\in\{1,\ldots,p-1\}$ and
 let $G$ be a gap of order $k$. 
 Then
 \begin{equation}\label{geq1}\sup_{x\in G} |f_0(x)-f(x)|+\sup_{x\in G} |Df_0(x)-Df(x)|\leq\varepsilon.\end{equation}

\item[-] (On gaps of high order)
Let $k\in\{n_2,\ldots,n-1\}$, $i\in\{1,\ldots,p-1\}$ and 
 consider the gap $(l,r)=X_{a_0\cdots a_{k-1}\Box_i}(h_{n-1})$ of order $k$. 
Let $\psi\colon [l,r]\to\mathbb R$ be a $C^1$ diffeomorphism for which 
the conclusion of Lemma \ref{smoothing} holds with $E=Df(l)$,
$E'=Df(r)$, $\tau=\frac{|f_0(l)-f_0(r)|}{l-r}$, $I=[l,r]$.
Then $f(x)=f_0(l)+\psi(x)$ holds
for every $x\in(l,r)$.
\end{itemize}

Let $a_0\cdots a_{n-1}\in W_n$.
From the definition of $f$ and \eqref{lipeq1}  
we have
\begin{equation}\label{geq4}\sup_{x\in X_{a_0\cdots a_{n-1} }(f_0)}|f_0(x)-f(x)|\leq  \left|X_{a_1\cdots a_{n-1}}(f_0)\right|\leq\frac{\varepsilon}{2}.\end{equation}
From the definition of $f$ and 
 \eqref{lipeq2} we have
\begin{equation}\label{geq5}\sup_{x\in X_{a_0\cdots a_{n-1}}(f_0)}|Df_0(x)-Df(x)|\leq\frac{\varepsilon}{4}.\end{equation}
Let $k\in\{ n_2,\ldots,n-1\}$, $i\in\{1,\ldots,p-1\}$ and 
 consider the gap $X_{a_0\cdots a_{k-1}\Box_i}(h_{n-1})$ of order $k$. 
Let $x\in X_{a_0\cdots a_{k-1}\Box_i}(h_{n-1})$.
 From \eqref{lipeq1},
\begin{equation}\label{geq2}|f_0(x)-f(x)|\leq\left| f_0(X_{a_0\cdots a_{k-1}\Box_i}(h_{n-1}))\right|\leq\frac{\varepsilon}{2}.\end{equation}
From Lemma \ref{smoothing} and \eqref{lipeq2}, 
\begin{equation}\label{geq3}|Df_0(x)-Df(x)|\leq||Df_0(x)|-\tau|+|\tau-|Df(x)||\leq\frac{\varepsilon}{4}+\frac{\varepsilon}{4}.\end{equation}
From \eqref{geq1}, \eqref{geq4}, \eqref{geq5}, \eqref{geq2}, \eqref{geq3}, $\|f_0-f\|_{C^1}\leq\varepsilon$ follows.
By Lemma \ref{standard}, (E2) holds for $f$.
Hence $f\in\mathscr{E}$.
(E3) follows from the definition of $f$.
\end{proof}

\subsection{Lipschitz continuity of $Df_\infty$}\label{LIP}
The next lemma is a version of the Realization Lemma.

\begin{lemma}\label{liplemma}
Let $f_0\in \mathscr{E}$ satisfy (E3) and 
 let $\theta\in(0,1/\sup_{x\in X}|Df_0(x)|)$.
  For every $\varepsilon>0$ there exists a neighborhood $\mathscr U$ of
  $\log|Df_0|\circ\pi_{f_0}$ in $C(\Sigma_p)$ 
such that for every $\phi\in \mathscr{U}\cap{\rm Lip}(\Sigma_p,d_\theta)$ 
there exists $f_\infty\in\mathscr{E}_{\rm Lip}$ such that
$$\|f_0-f_\infty\|_{C^1}\leq\varepsilon\ \ \text{and}\ \ \log|Df_\infty|\circ\pi_{f_\infty}=\phi.$$
\end{lemma}
\begin{proof}
Let $\varepsilon>0$ be such that
$\theta\leq1/(\sup_{x\in X}|Df_0(x)|+\varepsilon)$.
Although $f_0$ may not be of class $C^2$, 
from (E3) the estimate as in Lemma \ref{distortion} remains to hold,
and thus
the construction in Sect.\ref{PL} remains to work:
there exists a neighborhood $\mathscr U$ of 
$\log|Df_0|\circ\pi_{f_0}$ in $C(\Sigma_p)$ 
such that for every
$\phi\in \mathscr{U}$ there exists a Cauchy sequence $\{f_n\}_{n\geq0}$ in $\mathscr{E}$ 
such that $\|f_0-f_n\|_{C^1}\leq\varepsilon$ for every $n\geq1$ and
its $C^1$ limit $f_\infty$ belongs to $\mathscr{E}$
and satisfies
and $\log|Df_\infty|\circ\pi_{f_\infty}=\phi.$
We have
$\sup_{x\in X}|Df_n(x)|\leq  \sup_{x\in X}|Df_0(x)|+\varepsilon$.
Let $a_0\cdots a_{n-1}\in W_n$ and $i\in\{1,2,\ldots,p-1\}$.
From $X_{a_0\cdots a_{n-1}\widehat i}(h_{n})=X_{a_0\cdots a_{n-1}\widehat i}(f_{n})$,
$f^n\left(X_{a_0\cdots a_{n-1}\widehat i}(f_{n})\right)=[\alpha_{i-1},\beta_i]$
and the Mean Value Theorem we have
$$\left|X_{a_0\cdots a_{n-1}\widehat i}(h_{n})\right|\geq\frac{\left|f^n\left(X_{a_0\cdots a_{n-1}\widehat i}(f_{n})\right)\right|}{(\sup_{x\in X}|Df_n(x)|)^n}
\geq\frac{\beta_i-\alpha_{i-1}}{(\sup_{x\in X}|Df_0(x)|+\varepsilon)^n}.$$
Let $\phi\in{\rm Lip}(\Sigma_p,d_\theta)$.
Let
 $L(\phi)$ denote the Lipschitz constant of $\phi$.
Then
$$V_{a_0\cdots a_{n-1}}(\phi)\leq M(e^{\phi})L(\phi)\theta^n,$$
and therefore
$$\frac{V_{a_0\cdots a_{n-1}}(\phi)}{\left|X_{a_0\cdots a_{n-1}\widehat i}(h_{n})\right|}\leq
\frac{M(e^{\phi})L(\phi)}{\beta_i-\alpha_{i-1}}\theta^n(\sup_{x\in X}|Df_0(x)|+\varepsilon)^n\leq
\frac{M(e^{\phi})L(\phi)}{\beta_i-\alpha_{i-1}}.$$
Lemma \ref{d-ini} implies
$$\left|X_{a_0\cdots a_{n-1}\Box_i}(h_{n})\right|\geq \frac{1}{2K(f_0)+1}\left|X_{a_0\cdots a_{n-1}\widehat i}(h_{n})\right|.$$
Hence
$$\frac{V_{a_0\cdots a_{n-1}}(\phi)}{\left|X_{a_0\cdots a_{n-1}\Box_i}(h_{n})\right|}\leq(2K(f_0)+1)\frac{V_{a_0\cdots a_{n-1}}(\phi)}
{\left|X_{a_0\cdots a_{n-1}\widehat i}(h_{n})\right|}
\leq (2K(f_0)+1)\frac{M(e^{\phi})L(\phi)}{\beta_i-\alpha_{i-1}}.$$
From this estimate and the construction of $\{f_n\}_{n\geq0}$ in Sect.\ref{perturb}  
it follows that the Lipschitz constant of $Df_n$ is uniformly 
bounded over all $n$. 
The Lipschitz continuity of $Df_\infty$ is a direct consequence of the uniform convergence
of $Df_n$ to $Df_\infty$.
 \end{proof}

\subsection{End of proof of Theorem B}\label{endB}
To finish, we need an auxiliary lemma.

\begin{lemma}\label{lip}
Let $f_0\in\mathscr{E}$. There exist $\theta\in(0,1)$ and a neighborhood $\mathscr{W}$ of $f_0$ in $\mathscr{E}$ such that
 for every $f\in\mathscr{W}$, $\pi_f\colon\Sigma_p\to\Lambda(f)$ is Lipschitz continuous with respect to the distance $d_\theta$.
\end{lemma}
\begin{proof}
By Lemma \ref{standard} there exist constants $c>0$, $\lambda>1$ and a neighborhood $\mathscr{W}$ of $f_0$ in
 $\mathscr{E}$ such that the following holds for every $f\in\mathscr{W}$:
if $n\geq1$ and $x\in\bigcap_{k=0}^{n-1}f^{-k}(X)$, then $|Df^n(x)|\geq c\lambda^n$.
Let $\theta\in[\lambda^{-1},1)$.
For every $\underline{a}$, $\underline{b}\in\Sigma_p$, $\underline{a}\neq\underline{b}$ we have
$|\pi_f(\underline{a})-\pi_f(\underline{b})|\leq c\lambda^{-s(\underline{a},\underline{b})}\leq c\theta^{s(\underline{a},\underline{b})}=cd_\theta(\underline{a},\underline{b}).$
\end{proof}

 \begin{proof}[Proof of Theorem B]
Set $$\mathscr{O}=\left\{f\in\mathscr{E}_{\rm Lip}\colon \log|Df|\circ\pi_f\in\bigcup_{0<\theta<1}\mathscr{O}_{(\Sigma_p,d_\theta)}\right\}.$$
From Theorem 2, for maps in $\mathscr{O}$ the Lyapunov minimizing measure is unique, and it is supported on a periodic orbit.
From Lemma \ref{lip}, $\mathscr{O}$
 is an open subset of $\mathscr{E}_{\rm Lip}$.

  It remains to show $\mathscr{O}$ is dense in $\mathscr{E}$.
  Let
 $f_0\in\mathscr{E}$ satisfy (E3).
 By virtue of Lemma \ref{lc} it suffices to show that $f_0$ is approximated by elements of $\mathscr{O}$.
 Let $\theta\in(0,1/\sup_{x\in X}|Df_0(x)|)$.
 Since $\phi_0=\log|Df_0|\circ\pi_{f_0}$ is a locally constant function,
  $\phi_0\in {\rm Lip}(\Sigma_p,d_\theta)$ holds.
 By Lemma \ref{liplemma}, for every $\varepsilon>0$
 there exists a neighborhood $\mathscr{U}$ of $\phi_0$ in $C(\Sigma_p)$ 
  such that for every $\phi\in\mathscr{U}\cap   {\rm Lip}(\Sigma_p,d_\theta) $ 
  there exists $f_\infty\in\mathscr{E}_{\rm Lip}$ such that
 $\|f_0-f_\infty\|_{C^1}\leq\varepsilon$ and $\log|Df_\infty|\circ\pi_{f_\infty}=\phi$.
Since $\mathscr{O}_{(\Sigma_p,d_\theta)}$ is a dense subset of ${\rm Lip}(\Sigma_p,d_\theta)$ from Theorem 2,
 $\mathscr{U}\cap  \mathscr{O}_{(\Sigma_p,d_\theta)}\neq\emptyset$ holds.
If $\phi\in\mathscr{U}\cap  \mathscr{O}_{(\Sigma_p,d_\theta)}$ then
$f_\infty\in\mathscr{O}$.
  \end{proof}

\end{document}